\documentclass{amsart}
\usepackage{amssymb,latexsym,enumitem,graphics}
\pdfoutput=1

\theoremstyle{cute}
\newtheorem{thm}{Theorem}
\newtheorem{cor}[thm]{Corollary}

\newtheorem{prop}[thm]{Proposition}

\theoremstyle{definition}

\newtheorem{ex}[thm]{Example}

\theoremstyle{remark}

\numberwithin{equation}{section}

\begin{document}
\title{The Case for Raabe's Test}
\markright{The Case for Raabe's Test}
\author[C. N. B. Hammond]{Christopher N. B. Hammond}

\date{January 24, 2019}
\address{Department of Mathematics and Statistics\\
Connecticut College\\
New London, CT 06320}
\email{cnham@conncoll.edu}

\maketitle
\thispagestyle{empty}

\begin{abstract}
Among the techniques for determining the convergence of a series, Raabe's Test remains relatively unfamiliar to most mathematicians.  We present several results relating to Raabe's Test that do not seem to be widely known, making the case that Raabe's Test should be featured more prominently in undergraduate calculus and analysis courses.  In particular, we demonstrate that Raabe's Test may be viewed as an implicit comparison with a $p$-series, in the same manner that the Ratio Test and the Root Test constitute an implicit comparison with a geometric series.  Moreover, Raabe's Test can sometimes simplify the process for determining conditional convergence.
\end{abstract}

\section{Raabe's Test}\label{S:intro}

Although Raabe's Test was first introduced  in 1832, its importance and interpretation have largely been overlooked.  The purpose of this article is to expand the scope of Raabe's Test, illustrating its benefits and situating it within its proper context.

As most readers are probably aware, the Ratio Test and the Root Test can both be viewed as an implicit comparison with a geometric series; that is, they tell us when a series ``behaves like" a certain geometric series. Likewise, the version of Raabe's Test that we are presenting will indicate when a series ``behaves like" $\sum_{n=1}^{\infty}1/n^{p}$ for a particular $p$.  In a sense, one can view Raabe's Test as a type of comparison test that self-selects the appropriate $p$-series with which to compare.

Indeed, it is not always obvious when a series is comparable to a $p$-series.  For example, consider
\begin{equation}\label{firstseries}
\sum_{n=1}^{\infty}\frac{(-1)^{n-1}(2n)!}{4^{n}(n!)^{2}}\text{.}
\end{equation}
The standard convergence tests are not especially helpful in this instance. The Ratio Test and the Root Test are both inconclusive.  The Comparison and Limit Comparison Tests, besides applying only to series with non-negative terms, require a predetermined series with which to compare.  It is not even clear whether the hypotheses of the Alternating Series Test are satisfied.  Our version of Raabe's Test will show that this series essentially behaves like
\[
\sum_{n=1}^{\infty}\frac{(-1)^{n-1}}{n^{1/2}}\text{.}
\]
Hence, with only a single computation, we will demonstrate that series (\ref{firstseries}) is conditionally convergent.  (See Example \ref{goodex} below.)
In other words, Raabe's Test will allow us to perform a comparison without knowing beforehand to what we are comparing.  Moreover, the test will sometimes eliminate the need for employing a multistep process to determine conditional convergence.\bigskip

The original version of Raabe's Test, as stated by Joseph Ludwig Raabe \cite{raabe}, says that a series $\sum_{n=1}^{\infty}a_{n}$ consisting of positive terms converges whenever
\[
\lim_{n\to\infty}n\left(\frac{a_{n}}{a_{n+1}}-1\right)>1
\]
and diverges whenever
\[
\lim_{n\to\infty}n\left(\frac{a_{n}}{a_{n+1}}-1\right)<1\text{.}
\]
There are several minor variants of Raabe's result (see \cite[p.\ 39]{bromwich}, \cite[p.\ 285]{knopp}, or \cite{prus}), but the test is typically stated for series with only positive terms.  Raabe's Test can provide more information if we consider series that include both positive and negative terms, as illustrated by this slightly subtler version of the test.

\begin{thm}[Raabe's Test]\label{raabestest}
Suppose $\sum_{n=1}^{\infty}a_{n}$ is a series consisting of nonzero terms, for which
\[
p=\lim_{n\to\infty}n\left(\left|\frac{a_{n}}{a_{n+1}}\right|-1\right)
\]
exists (as a finite value).  If $p>1$, the series converges absolutely.  If $p<0$, the series diverges.  If $0\leq p<1$, the series is either conditionally convergent or divergent.  If $p=1$, the test provides no information.
\end{thm}

\begin{proof}
Suppose, first of all, that $p>1$.  Since $(p-1)/2$ is a positive number, there is a natural number $N$ such that
\[
\left|n\left(\left|\frac{a_{n}}{a_{n+1}}\right|-1\right)-p\right|<\frac{p-1}{2}\text{,}
\]
and hence
\[
n\left(\left|\frac{a_{n}}{a_{n+1}}\right|-1\right)>p-\frac{p-1}{2}=\frac{p+1}{2}\text{,}
\]
whenever $n\geq N$.  Let $p_{1}=(p+1)/2$, which is also greater than $1$.  Observe that
\[
n\left|\frac{a_{n}}{a_{n+1}}\right|>p_{1}+n
\]
and thus
\[
n|a_{n}|>(p_{1}+n)|a_{n+1}|
\]
for $n\geq N$, from which we see that
\begin{equation}\label{raabeprime}
n|a_{n}|-(n+1)|a_{n+1}|>(p_{1}-1)|a_{n+1}|\text{.}
\end{equation}
Since $p_{1}-1$ is positive, it follows that
\[
n|a_{n}|>(n+1)|a_{n+1}|
\]
for $n\geq N$.  Since every term $n|a_{n}|$ is positive, the Monotone Convergence Theorem guarantees that the sequence $\bigl(n|a_{n}|\bigr)$ converges to some limit $x$.  Consider the series
\[
\sum_{n=1}^{\infty}b_{n}=\sum_{n=1}^{\infty}\bigl(n|a_{n}|-(n+1)|a_{n+1}|\bigr)\text{.}
\]
The $m$th partial sum of $\sum_{n=1}^{\infty}b_{n}$ is equal to
\begin{align*}
\bigl(|a_{1}|-2|a_{2}|\bigr)&+\bigl(|2a_{2}|-3|a_{3}|\bigr)+\ldots+\bigl(m|a_{m}|-(m+1)|a_{m}|\bigr)\\
&=|a_{1}|-(m+1)|a_{m+1}|\text{,}
\end{align*}
so the series converges to $|a_{1}|-x$.  Therefore the Comparison Test, along with (\ref{raabeprime}), shows that
\[
\sum_{n=1}^{\infty}(p_{1}-1)|a_{n+1}|
\]
is convergent, as is
\[
\sum_{n=1}^{\infty}|a_{n+1}|=\frac{1}{p_{1}-1}\sum_{n=1}^{\infty}(p_{1}-1)|a_{n+1}|\text{.}
\]
Consequently the series $\sum_{n=1}^{\infty}a_{n}$ converges absolutely.

Now suppose that $p<0$.  Since $-p$ is a positive number, there is a natural number $N$ such that
\[
\left|n\left(\left|\frac{a_{n}}{a_{n+1}}\right|-1\right)-p\right|<-p\text{,}
\]
and hence
\[
n\left(\left|\frac{a_{n}}{a_{n+1}}\right|-1\right)<p-p=0\text{,}
\]
whenever $n\geq N$.  Consequently
\[
\left|\frac{a_{n}}{a_{n+1}}\right|-1<0
\]
for $n\geq N$, which means that $|a_{n}|<|a_{n+1}|$ whenever $n\geq N$.  Therefore the sequence $(a_{n})$ does not converge to $0$, so the series $\sum_{n=1}^{\infty}a_{n}$ is divergent.

Finally, suppose that $0\leq p<1$.  Since $(1-p)/2$ is a positive number, there is a natural  number $N$ such that
\[
\left|n\left(\left|\frac{a_{n}}{a_{n+1}}\right|-1\right)-p\right|<\frac{1-p}{2}\text{,}
\]
and hence
\[
n\left(\left|\frac{a_{n}}{a_{n+1}}\right|-1\right)<p+\frac{1-p}{2}=\frac{1+p}{2}\text{,}
\]
whenever $n\geq N$.  Therefore
\[
n\left|\frac{a_{n}}{a_{n+1}}\right|-(n+1)<\frac{1+p}{2}-1=\frac{p-1}{2}<0
\]
for $n\geq N$.  Consequently
\[
n|a_{n}|<(n+1)|a_{n+1}|
\]
for $n\geq N$, which means that $N|a_{N}|\leq n |a_{n}|$ whenever $n\geq N$.  Taking $M=N|a_{N}|$, we see that
\[
\frac{M}{n}\leq|a_{n}|
\]
whenever $n\geq N$.  Thus the Comparison Test shows that $\sum_{n=1}^{\infty}|a_{n}|$ is divergent, so $\sum_{n=1}^{\infty}a_{n}$ is either conditionally convergent or divergent.

Example \ref{raabeex}, which appears in the next section, will demonstrate that any outcome is possible when $p=1$.  Namely, such a series may diverge, may converge conditionally, or may converge absolutely.
\end{proof}

Observe that the quantity
\[
p=\lim_{n\to\infty}n\left(\left|\frac{a_{n}}{a_{n+1}}\right|-1\right)
\]
from Raabe's Test cannot exist unless
\[
\lim_{n\to\infty}\left(\left|\frac{a_{n}}{a_{n+1}}\right|-1\right)=0\text{.}
\]
Therefore
\[
\lim_{n\to\infty}\left|\frac{a_{n+1}}{a_{n}}\right|
\]
must exist and be equal to $1$.  In other words, Raabe's Test presupposes the inconclusiveness of the Ratio Test.  By Theorem 3.37 in \cite{rudin}, the existence of a finite value of $p$ also guarantees that the Root Test is inconclusive.\bigskip

As mentioned above, Raabe's Test yields no information when $p=1$.  For any value $0\leq p<1$, there are examples for which the series converges conditionally and examples for which the series diverges.  (See Example \ref{raabeex} below.)  In general, it can be difficult to determine the behavior of a series when $0\leq p\leq 1$, although the next two results are rather helpful.

\begin{prop}\label{raabezeroprop}
Suppose $(a_{n})$ is a sequence consisting of nonzero numbers, for which
\[
p=\lim_{n\to\infty}n\left(\left|\frac{a_{n}}{a_{n+1}}\right|-1\right)
\]
exists. If $p>0$, the sequence $(a_{n})$ converges to $0$.
\end{prop}

\begin{proof}
Since $p/2$ is a positive number, there is a natural number $N$ such that
\[
\left|n\left(\left|\frac{a_{n}}{a_{n+1}}\right|-1\right)-p\right|<\frac{p}{2}\text{,}
\]
and hence
\begin{equation}\label{raabealt}
n\left(\left|\frac{a_{n}}{a_{n+1}}\right|-1\right)>p-\frac{p}{2}=\frac{p}{2}>0\text{,}
\end{equation}
whenever $n\geq N$.  Consequently
\[
\left|\frac{a_{n}}{a_{n+1}}\right|-1>0
\]
for $n\geq N$, so $|a_{n}|>|a_{n+1}|$ whenever $n\geq N$.  The Monotone Convergence Theorem guarantees that $(|a_{n}|)$ converges to a non-negative number $x$.  We simply need to show that $x$ is equal to $0$.

Suppose, for the sake of contradiction, that $x>0$.  Since $|a_{n}|>x$ for $n\geq N$, it follows from (\ref{raabealt}) that
\[
\frac{n(|a_{n}|-|a_{n+1}|)}{x}>n\left(\frac{|a_{n}|-|a_{n+1}|}{|a_{n+1}|}\right)=n\left(\left|\frac{a_{n}}{a_{n+1}}\right|-1\right)>\frac{p}{2}
\]
when $n\geq N$.  Therefore
\[
|a_{n}|-|a_{n+1}|>\frac{px}{2n}\text{,}
\]
and hence
\begin{align*}
|a_{m}|-|a_{n}|&=(|a_{m}|-|a_{m+1}|)+(|a_{m+1}|-|a_{m+2}|)+\ldots+(|a_{n-1}|-|a_{n}|)\\
&>\frac{px}{2}\left(\frac{1}{m}+\frac{1}{m+1}+\ldots+\frac{1}{n-1}\right)\text{,}
\end{align*}
whenever $n>m\geq N$.  While $(|a_{n}|)$ is a Cauchy sequence, the sequence of partial sums for the harmonic series is divergent.  Thus we have obtained a contradiction, so we conclude that $x$ must be $0$.
\end{proof}

In certain instances, such as series (\ref{firstseries}), it is reasonable to use Proposition \ref{raabezeroprop} to show that $(a_{n})$ converges to $0$. (See also Problem 2047 in \cite{problems}.)  The main purpose of this proposition, though, is to expand the scope of Raabe's Test.

\begin{thm}[Raabe's Alternating Series Test]\label{rast}
Suppose $(a_{n})$ is a sequence consisting of positive numbers, for which
\[
p=\lim_{n\to\infty}n\left(\frac{a_{n}}{a_{n+1}}-1\right)
\]
exists.  If $0<p\leq 1$, the series
\[
\sum_{n=1}^{\infty}(-1)^{n-1}a_{n}=a_{1}-a_{2}+a_{3}-a_{4}+\ldots
\]
is convergent.  If $0<p<1$, the series converges conditionally.
\end{thm}

\begin{proof}
In addition to showing that $(a_{n})$ converges to $0$, the proof of Proposition \ref{raabezeroprop} shows that there is a natural number $N$ such that $a_{n}>a_{n+1}$ for $n\geq N$.  Hence the Alternating Series Test guarantees that the series stated above is convergent.  If $0<~p<1$, Raabe's Test shows that the series must, in fact, be conditionally convergent.
\end{proof}

Theorem \ref{rast} is somewhat unusual, in that it typically takes multiple steps to show that a series is conditionally convergent.

\begin{ex}\label{goodex}
Consider the series
\[
\sum_{n=1}^{\infty}\frac{(-1)^{n-1}(2n-1)!}{4^{n}(n!)^{2}}
\]
and
\[
\sum_{n=1}^{\infty}\frac{(-1)^{n-1}(2n)!}{4^{n}(n!)^{2}}
\]
and
\[
\sum_{n=1}^{\infty}\frac{(-1)^{n-1}(2n+1)!}{4^{n}(n!)^{2}}\text{,}
\]
the second of which was mentioned at the beginning of this article.  It is not difficult to demonstrate that these series have values $3/2$, $1/2$, and $-1/2$ with respect to Raabe's Test.  Therefore the first series converges absolutely, the second series converges conditionally, and the third series diverges.
\end{ex}

\section{Interpretation and Examples}

Our next observation is fundamental to our interpretation of Raabe's Test.

\begin{ex}\label{pseriesex}
Let $a_{n}=1/n^{p}$ for a real number $p$. Observe that
\[
n\left(\left|\frac{a_{n}}{a_{n+1}}\right|-1\right)=n\left(\left(\frac{n+1}{n}\right)^{p}-1\right)
=\frac{\left(1+\displaystyle\frac{1}{n}\right)^{p}-1}{\displaystyle\frac{1}{n}}\text{,}
\]
which converges to $f^{\prime}(1)$ for $f(x)=x^{p}$. Consequently the expression above converges to $p$.
\end{ex}

As we have previously mentioned, Raabe's Test should be viewed as an implicit comparison between $\sum_{n=1}^{\infty}|a_{n}|$ and the corresponding series $\sum_{n=1}^{\infty}1/n^{p}$.  In other words, if $\sum_{n=1}^{\infty}a_{n}$ has value $p$ with respect to Raabe's Test, the series behaves as if $|a_{n}|=1/n^{p}$. The only exceptions are the borderline cases where $p=0$ and $p=1$.  Note that the proof of Raabe's Test did not require any prior knowledge about $p$-series, except for the fact that the harmonic series diverges.

Viewing Raabe's Test from this perspective, one would anticipate certain behavior with respect to products and powers.  If $\sum_{n=1}^{\infty}|a_{n}|$ and $\sum_{n=1}^{\infty}|b_{n}|$ behave like $\sum_{n=1}^{\infty}1/n^{p}$ and $\sum_{n=1}^{\infty}1/n^{q}$, then $\sum_{n=1}^{\infty}|a_{n}b_{n}|$ should behave like $\sum_{n=1}^{\infty}1/n^{p+q}$ and $\sum_{n=1}^{\infty}|a_{n}|^{k}$ should behave like $\sum_{n=1}^{\infty}1/n^{kp}$.  The following proposition formalizes this intuition.

\begin{prop}\label{raabealg}
Suppose $(a_{n})$ and $(b_{n})$ are sequences consisting of nonzero numbers, for which
\[
p=\lim_{n\to\infty}n\left(\left|\frac{a_{n}}{a_{n+1}}\right|-1\right)
\]
and
\[
q=\lim_{n\to\infty}n\left(\left|\frac{b_{n}}{b_{n+1}}\right|-1\right)
\]
both exist.  In that case,
\[
\lim_{n\to\infty}n\left(\left|\frac{a_{n}b_{n}}{a_{n+1}b_{n+1}}\right|-1\right)
\]
exists and is equal to $p+q$.  Furthermore, for any real number $k$, the expression
\[
\lim_{n\to\infty}n\left(\left|\frac{a_{n}}{a_{n+1}}\right|^{k}-1\right)
\]
exists and is equal to $kp$.
\end{prop}

\begin{proof}
First of all, note that
\begin{align*}
n\left(\left|\frac{a_{n}b_{n}}{a_{n+1}b_{n+1}}\right|-1\right)&=n\left(\frac{|a_{n}b_{n+1}|}{|a_{n+1}b_{n+1}|}+\frac{|a_{n}b_{n}|-|a_{n}b_{n+1}|}{|a_{n+1}b_{n+1}|}-1\right)\\
&=n\left(\left|\frac{a_{n}}{a_{n+1}}\right|-1+\left|\frac{a_{n}}{a_{n+1}}\right|\left(\frac{|b_{n}|-|b_{n+1}|}{|b_{n+1}|}\right)\right)\\
&=n\left(\left|\frac{a_{n}}{a_{n+1}}\right|-1\right)+\left|\frac{a_{n}}{a_{n+1}}\right|n\left(\left|\frac{b_{n}}{b_{n+1}}\right|-1\right)\text{.}
\end{align*}
Since
\[
\lim_{n\to\infty}\left|\frac{a_{n+1}}{a_{n}}\right|=1\text{,}
\]
our first assertion follows from the Algebraic Limit Theorem.

To prove the second assertion, consider the continuous function
\[
g(x)=\begin{cases}
\displaystyle\frac{x^{k}-1}{x-1}, & x\neq 1\\
k, & x=1
\end{cases}\text{.}
\]
Observe that
\[
n\left(\left|\frac{a_{n}}{a_{n+1}}\right|^{k}-1\right)=n\left(\left|\frac{a_{n}}{a_{n+1}}\right|-1\right)g\!\left(\left|\frac{a_{n}}{a_{n+1}}\right|\right)
\]
for all $n$.  Since $|a_{n}/a_{n+1}|$ is converging to $1$, the expression above converges to the product of $p$ and $k$.
\end{proof}

In other words, if $\sum_{n=1}^{\infty}a_{n}$ and $\sum_{n=1}^{\infty}b_{n}$ have values $p$ and $q$ with respect to Raabe's Test, then $\sum_{n=1}^{\infty}a_{n}b_{n}$ has value $p+q$.  Likewise, whenever it is defined, $\sum_{n=1}^{\infty}a_{n}^{k}$ has value $kp$.  In particular, $\sum_{n=1}^{\infty}1/a_{n}$ has value $-p$ and $\sum_{n=1}^{\infty}a_{n}/b_{n}$ has value $p-q$.  Besides confirming our interpretation of Raabe's Test, these observations can be quite useful from a computational perspective

As noted by Knopp \cite[p.\ 287]{knopp}, there is an equivalent formulation of Raabe's Test:
\begin{equation}\label{schlomo}
p=\lim_{n\to\infty}n\log\left|\frac{a_{n}}{a_{n+1}}\right|\text{,}
\end{equation}
where $\log$ denotes the natural logarithm.  (The equivalence of (\ref{schlomo}) to the standard form of Raabe's Test can be deduced from the inequalities $(x-1)/x\leq\log x\leq x-1$.) This version, which is sometimes called \textit{Schl\"{o}milch's Test} \cite{prus2}, makes the results of Proposition \ref{raabealg} seem a bit more apparent.
\bigskip

If $p>0$, Proposition \ref{raabezeroprop} shows that $(a_{n})$ must converge to $0$.  Hence it follows from Proposition \ref{raabealg} that $(a_{n})$ is unbounded whenever $p<0$.  Let us pause for a moment to consider the case where $p=0$.  
From the perspective of Raabe's Test, such a series ``looks like" the $p$-series
\[
\sum_{n=1}^{\infty}\frac{1}{n^{0}}=1+1+1+\cdots\text{.}
\]
Nevertheless, there are examples with $p=0$ for which $(a_{n})$ is either convergent to $0$ or unbounded.

\begin{ex}\label{logex}
Taking $a_{n}=1/\log(n+1)$, we see that
\begin{align}
n\left(\left|\frac{a_{n}}{a_{n+1}}\right|-1\right)&=n\left(\frac{\log(n+2)}{\log(n+1)}-1\right)\nonumber\\
&=\frac{n\bigl(\log(n+2)-\log(n+1)\bigr)}{\log(n+1)}\nonumber\\
&=\frac{\log\!\left(\bigl(1+\frac{1}{n+1}\bigr)^{n}\right)}{\log(n+1)}\text{.}\label{logquot}
\end{align}
Since the numerator of (\ref{logquot}) converges to $\log e=1$, the entire expression converges to $0$.  Proposition \ref{raabealg} shows that
\[
\lim_{n\to\infty}n\left(\left|\frac{b_{n}}{b_{n+1}}\right|-1\right)
\]
is also $0$, where $b_{n}=\log(n+1)$.
\end{ex}

In other words, there is some ``wiggle room" with respect to how closely a series must resemble the corresponding series $\sum_{n=1}^{\infty}1/n^{p}$.  We are now in a position to justify the final assertion in the statement of Raabe's Test, as well as a remark made shortly before the statement of Proposition \ref{raabezeroprop}.

\begin{ex}\label{raabeex}
We would like to illustrate the range of possible outcomes when $0\leq p\leq 1$.  First of all, consider the case where $p=0$.  The Alternating Series Test shows that
\[
\sum_{n=1}^{\infty}\frac{(-1)^{n-1}}{\log(n+1)}
\]
is conditionally convergent.  On the other hand, the series
\[
\sum_{n=1}^{\infty}\frac{1}{n^{0}}=1+1+1+\cdots
\]
is divergent.

For $0<p\leq 1$, the series
\[
\sum_{n=1}^{\infty}\frac{(-1)^{n-1}}{n^{p}}
\]
is conditionally convergent and has value $p$ with respect to Raabe's test.  Similarly, the series
\[
\sum_{n=1}^{\infty}\frac{1}{n^{p}}
\]
is divergent. 

We still need to identify a series with $p=1$ that is absolutely convergent.  It follows from Proposition \ref{raabealg} that
\[
\sum_{n=1}^{\infty}\frac{1}{(n+1)\bigl(\log(n+1)\bigr)^{2}}
\]
has value $1$ with respect to Raabe's Test.  Moreover, the Integral Test shows that this series converges absolutely.
\end{ex}

The next result is particularly useful when applying Raabe's Test to concrete examples.

\begin{prop}\label{sumprop}
Suppose $(a_{n})$ and $(b_{n})$ are sequences consisting of positive numbers, for which
\[
p=\lim_{n\to\infty}n\left(\frac{a_{n}}{a_{n+1}}-1\right)
\]
and
\[
q=\lim_{n\to\infty}n\left(\frac{b_{n}}{b_{n+1}}-1\right)
\]
both exist. If $p<q$, then
\begin{equation}\label{rsum}
\lim_{n\to\infty}n\left(\frac{a_{n}+b_{n}}{a_{n+1}+b_{n+1}}-1\right)
\end{equation}
and
\begin{equation}\label{rdiff}
\lim_{n\to\infty}n\left(\left|\frac{a_{n}-b_{n}}{a_{n+1}-b_{n+1}}\right|-1\right)
\end{equation}
both exist and are equal to $p$.
\end{prop}

\begin{proof} Note that
\begin{align*}
n\left(\frac{a_{n}+b_{n}}{a_{n+1}+b_{n+1}}-1\right)=n\left(\frac{a_{n}-a_{n+1}}{a_{n+1}+b_{n+1}}+\frac{b_{n}-b_{n+1}}{a_{n+1}+b_{n+1}}\right)&\\
=n\left(\frac{a_{n}}{a_{n+1}}-1\right)\left(\frac{1}{1+\displaystyle\frac{b_{n+1}}{a_{n+1}}}\right)+n\left(\frac{b_{n}}{b_{n+1}}-1\right)\left(\frac{\displaystyle\frac{b_{n+1}}{a_{n+1}}}{1+\displaystyle\frac{b_{n+1}}{a_{n+1}}}\right)\text{.}&
\end{align*}
The sequence $(b_{n}/a_{n})$ has value $q-p>0$ with respect to Raabe's Test, so Proposition \ref{raabezeroprop} dictates that $(b_{n}/a_{n})$ converges to $0$.  Thus the expression above converges to $p\cdot 1+q\cdot 0=p$.

Since $(b_{n}/a_{n})$ converges to $0$, there is a natural number $N$ such that $b_{n}<a_{n}$ for $n\geq N$.  Consequently $|a_{n}-b_{n}|=a_{n}-b_{n}$ for $n\geq N$, and hence
\begin{align*}
n\left(\left|\frac{a_{n}-b_{n}}{a_{n+1}-b_{n+1}}\right|-1\right)=n\left(\frac{a_{n}-a_{n+1}}{a_{n+1}-b_{n+1}}-\frac{b_{n}-b_{n+1}}{a_{n+1}-b_{n+1}}\right)&\\
=n\left(\frac{a_{n}}{a_{n+1}}-1\right)\left(\frac{1}{1-\displaystyle\frac{b_{n+1}}{a_{n+1}}}\right)-n\left(\frac{b_{n}}{b_{n+1}}-1\right)\left(\frac{\displaystyle\frac{b_{n+1}}{a_{n+1}}}{1-\displaystyle\frac{b_{n+1}}{a_{n+1}}}\right)\text{.}&
\end{align*}
Therefore this expression also converges to $p$.
\end{proof}

The result above is not valid when $p=q$, since
\[
\frac{1}{n^{p}}-\left(\frac{1}{n^{p}}+\frac{1}{n^{p+1}}\right)=\frac{1}{n^{p+1}}\text{.}
\]
The analogous result may fail to hold for sequences consisting of nonzero terms, even when
\[
\lim_{n\to\infty}n\left(\left|\frac{a_{n}}{a_{n+1}}\right|-1\right)=p<q=\lim_{n\to\infty}n\left(\left|\frac{b_{n}}{b_{n+1}}\right|-1\right)\text{.}
\]
If $a_{n}=1$ and $b_{n}=(-1)^{n-1}/n$, for example, then both (\ref{rsum}) and (\ref{rdiff}) are undefined.\bigskip

For any non-negative integer $m$, Example \ref{pseriesex} shows that the monomial $n^{m}$ has value $-m$ with respect to Raabe's Test.  Hence Proposition \ref{sumprop} allows us to compute the value for any series whose terms can be expressed as a polynomial in $n$.

\begin{ex}\label{polyex}
Consider the series
\[
\sum_{n=1}^{\infty}6n^{4}-11n^{3}-3n^{2}+7n+5=\sum_{n=1}^{\infty}\bigl(6n^{4}+7n+5\bigr)-\bigl(11n^{3}+3n^{2}\bigr)\text{.}
\]
Proposition \ref{sumprop} shows that $\sum_{n=1}^{\infty}6n^{4}+7n+5$ and $\sum_{n=1}^{\infty}11n^{3}+3n^{2}$ have values $-4$ and $-3$ with respect to Raabe's Test.  Thus their difference, which is the original series, has value $-4$.
\end{ex}

The reasoning from this example leads us to the following observation.

\begin{cor}\label{polyraabe}
Let $\sum_{n=1}^{\infty}a_{n}$ be a series consisting of nonzero terms.  If $a_{n}=h(n)$ for a polynomial $h$ of degree $m$, the series has value $-m$ with respect to Raabe's Test.
\end{cor}

Combining the results of this section, we can often apply Raabe's Test without much additional computation.

\begin{ex}\label{limcom}
Consider the series
\[
\sum_{n=1}^{\infty}a_{n}=\sum_{n=1}^{\infty}(-1)^{n-1}\sqrt{\frac{n^{2}-2n+3}{5n^{3}-7n^{2}+11n+13}}\text{.}
\]
Corollary \ref{polyraabe} shows that the numerator of the expression inside the radical has value $-2$ with respect to Raabe's Test and that the denominator has value $-3$.  Thus Proposition \ref{raabealg} shows that their quotient has value $1$ with respect to Raabe's test and that $\sum_{n=1}^{\infty}a_{n}$ has value $1/2$.  Therefore the series converges conditionally.

We could obtain the same result by applying the Limit Comparison Test and the Alternating Series Test, comparing $\sum_{n=1}^{\infty}|a_{n}|$ with $\sum_{n=1}^{\infty}1/n^{1/2}$.  The informal process by which we obtain the series $\sum_{n=1}^{\infty}1/n^{1/2}$ requires essentially the same steps as actually employing Raabe's Test.  Moreover, Theorem \ref{rast} makes it unnecessary to verify the hypotheses of the Alternating Series Test.
\end{ex}

\section{Pedagogical Implications}

When first learning about series, students typically encounter two fundamental classes of examples:  geometric series and $p$-series.  By analogy, it makes sense to introduce students to the Ratio Test and the Root Test (which relate to geometric series) and also to Raabe's Test (which relates to $p$-series).  The most obvious way to feature Raabe's Test in a calculus or an analysis course would be as a supplement to the Limit Comparison Test.  There are several reasons that such an innovation might be advantageous:
\begin{itemize}
\item We generally use the Limit Comparison Test when the Ratio Test and the Root Test are inconclusive, which is precisely the situation in which Raabe's Test may be applicable.
\item When applying the Limit Comparison Test, we often compare the series in which we are interested with a $p$-series.  Raabe's Test serves a similar function, but does not require that we know the value of $p$ beforehand.   (As noted in Example \ref{limcom}, the process for finding a series with which to compare may involve essentially the same steps as applying Raabe's Test.)
\item In certain cases, Raabe's Test allows us immediately to draw conclusions about conditional convergence.
\end{itemize}

\noindent In the context of $p$-series, the one instance in which the Limit Comparison Test can yield more information than Raabe's Test is when $p=1$.  That case is inconclusive with respect to Raabe's Test, but a comparison with the harmonic series may demonstrate divergence. We might consider using the Limit Comparison Test in a more targeted manner, reserving it primarily for this situation, and applying Raabe's Test otherwise.

\section*{Acknowledgments}

The author is indebted to Warren P.\ Johnson for many helpful discussions during the preparation of this article.  The author would also like to thank Edward Omey for his thoughtful correspondence and the anonymous reviewers for their suggestions regarding the organization and structure of the article.

 \end{document}